\documentclass[runningheads]{llncs}
\usepackage{graphicx}

\usepackage{amssymb,amsmath}
\usepackage{amsrefs, float}
\usepackage{bbold,stackrel}
 
\newcommand{\Z}{{\textsf{\textup{Z}}}}

\newtheorem{thm}{Theorem}

\newtheorem{cor}[thm]{Corollary}
\newtheorem{defi}[thm]{Definition}
\newtheorem{rem}[thm]{Remark}
\newtheorem{nota}[thm]{Notation}

\newtheorem{princ}[thm]{Principle}

\newcommand\be{\begin{equation}}
\newcommand\ee{\end{equation}} 

\usepackage{amsmath,amsfonts} 

\usepackage[applemac]{inputenc}

\def\bdefi{\begin{defi}\rm}
\def\edefi{\end{defi}}
\def\bnota{\begin{nota}\rm}
\def\enota{\end{nota}}
\def\FIVE{\Pi_{1}^{1}\text{-\textup{\textsf{CA}}}_{0}}

\def\SIXK{\Pi_{k}^{1}\text{-\textsf{\textup{CA}}}_{0}^{\omega}}

\def\L{\textsf{\textup{L}}}

\def\RCA{\textup{\textsf{RCA}}}
\def\({\textup{(}}
\def\){\textup{)}}

\def\cc{\textup{\textsf{c}}}

\def\RCAo{\textup{\textsf{RCA}}_{0}^{\omega}}
\def\ACAo{\textup{\textsf{ACA}}_{0}^{\omega}}
\def\WKL{\textup{\textsf{WKL}}}

\def\WWKL{\textup{\textsf{WWKL}}}
\def\bye{\end{document}}
\def\N{{\mathbb  N}}
\def\Q{{\mathbb  Q}}
\def\R{{\mathbb  R}}
\def\A{{\textsf{\textup{A}}}}

\def\SS{\textup{\textsf{S}}}

\def\di{\rightarrow}

\def\asa{\leftrightarrow}

\def\ACA{\textup{\textsf{ACA}}}

\def\QFAC{\textup{\textsf{QF-AC}}}

\def\WHBC{\textup{\textsf{WHBC}}}

\def\NIN{\textup{\textsf{NIN}}}

\def\HBC{\textup{\textsf{HBC}}}

\def\BOOT{\textup{\textsf{BOOT}}}

\def\NFP{\textup{\textsf{NFP}}}

\def\HBU{\textup{\textsf{HBU}}}

\def\WHBU{\textup{\textsf{WHBU}}}

\def\eps{\varepsilon}

\def\ECF{\textup{\textsf{ECF}}}

\numberwithin{equation}{section}
\numberwithin{thm}{section}

\begin{document}
\title{Splittings and robustness for the Heine-Borel theorem\thanks{
Supported by the \emph{Deutsche Forschungsgemeinschaft} via the DFG grant SA3418/1-1.}}
%
%
\author{Sam Sanders\inst{1}}
\authorrunning{S.\ Sanders}
%
\institute{Institute for Philosophy II, RUB, Bochum, Germany\\ \email{sasander@me.com}}
%


\maketitle              
\begin{abstract}
The Heine-Borel theorem for uncountable coverings has recently emerged as an interesting and central principle in higher-order Reverse Mathematics and computability theory, formulated as follows:
$\HBU$ is the Heine-Borel theorem for uncountable coverings given as $\cup_{x\in [0,1]}(x-\Psi(x), x+\Psi(x))$ for arbitrary $\Psi:[0,1]\di \R^{+}$, i.e.\ the original formulation going back to Cousin (1895) and Lindel\"of (1903).   
In this paper, we show that $\HBU$ is equivalent to its restriction to functions \emph{continuous almost everywhere}, an elegant robustness result.  We also obtain a nice splitting $\HBU\asa [\WHBU^{+}+\HBC_{0}+\WKL]$ where $\WHBU^{+}$ is 
a strengthening of Vitali's covering theorem and where $\HBC_{0}$ is the Heine-Borel theorem for countable collections (and \textbf{not sequences}) of basic open intervals, as formulated by Borel himself in 1898.  
\end{abstract}
\section{Introduction and preliminaries}\label{intro}
We sketch our aim and motivation within the \emph{Reverse Mathematics} program (Section \ref{krum}) and introduce some essential axioms and definitions (Section \ref{prelim}).
\subsection{Aim and motivation}\label{krum}
Reverse Mathematics (RM hereafter) is a program in the foundations of mathematics initiated by Friedman (\cites{fried, fried2}) and developed extensively by Simpson and others (\cites{simpson1, simpson2}); an introduction to RM for the `mathematician in the street' may be found in \cite{stillebron}.  We assume basic familiarity with RM, including Kohlenbach's \emph{higher-order} RM introduced in \cite{kohlenbach2}.  
Recent developments in higher-order RM, including our own, are published in \cite{dagsamV, dagsamVI, dagsamIII, dagsamX, dagsamVII, dagsamIX, dagsamXI}.

\smallskip

Now, a \emph{splitting} $A\asa [B+C]$ is a relatively rare phenomenon in second-order RM where a natural theorem $A$ can be \emph{split} into two \emph{independent} natural parts $B$ and $C$.
Splittings are quite common in higher-order RM, as studied in some detail in \cite{samsplit}.  An unanswered question here is whether the higher-order generalisations of the Big Five of RM (and related principles) have natural splittings.  

\smallskip

In this paper, we study the Vitali and Heine-Borel covering theorems for uncountable coverings with an eye on splittings.  
In particular, our starting point is $\HBU$, defined in Section \ref{prelim}, which is the Heine-Borel theorem for uncountable coverings $\cup_{x\in [0,1]}I_{x}^{\Psi}$ for \emph{arbitrary} third-order $\Psi:[0,1]\di \R^{+}$ and $I_{x}^{\Psi}\equiv (x-\Psi(x), x+\Psi(x))$.  This kind of coverings was already studied by Cousin in 1895 (\cite{cousin1}) and Lindel\"of in 1903 (\cite{blindeloef}).  
In Section~\ref{X}, we obtain an elegant splitting involving $\HBU$, namely as follows:  
\be\label{baggin}
\HBU\asa [\WHBU^{+}+\HBC_{0}+\WKL ],
\ee
where $\WHBU^{+}$ is a strengthening of the Vitali covering theorem and where $\HBC_{0}$ is the Heine-Borel theorem for \emph{countable collections} (and \textbf{not} sequences) of open intervals, as formulated by Borel himself in \cite{opborrelen2}.
In Section \ref{Y}, we prove $\HBU\asa\HBU_{\ae}$, where the latter is $\HBU$ restricted to functions $\Psi:[0,1]\di\R^{+}$ continuous \emph{almost everywhere} on the unit interval.   By contrast, the same restriction for the Vitali covering theorem results in a theorem equivalent to \emph{weak weak K\"onig's lemma} $\WWKL$.   The results in Section \ref{Y} were obtained following the study of splittings involving `continuity almost everywhere'.    
The proof of Theorem \ref{frigi} (in a stronger system) was suggested to us by Dag Normann.   In general, this paper constitutes a spin-off from our joint project with Dag Normann on the Reverse Mathematics and computability theory of the uncountable (see \cites{dagsamIII, dagsamX, dagsamXI}). 

\smallskip

Finally, the foundational and historical significance of our results is as follows.
\begin{rem}\label{fundies}\rm
First of all, as shown in \cites{dagsamIII, dagsamV, dagsamVII}, the third-order statements $\HBU$ and $\WHBU$ cannot be proved $\Z_{2}^{\omega}$, a conservative extension of $\Z_{2}$ based on third-order comprehension functionals.
A sceptic of third-order objects could `downplay' this independence result by pointing to the outermost quantifier of $\HBU$ and $\WHBU$ and declare that the strength of these principles is simply due to the quantification 
over \emph{all} third-order functions.  This point is moot in light of $\HBU\asa \HBU_{\ae}$ proved in Theorem \ref{frigi}, and the central role of `continuity almost everywhere' in e.g.\ the study of the Riemann integral and measure theory.

\smallskip

Secondly, our first attempt at obtaining a splitting for $\HBU$ was to decompose the latter as $\HBU_{\ae}+\WHBU$, where $\WHBU$ allows one to reduce an arbitrary covering to a covering generated by a function that is continuous almost everywhere.  
Alas, this kind of splitting does not yield \emph{independent} conjuncts, which is why we resort to stronger notions like \emph{countability}, namely in Section~\ref{X}.  

\smallskip

Thirdly, the splitting in \eqref{baggin} has some historical interest as well: Borel himself formulates the Heine-Borel theorem in \cite{opborrelen2} using \emph{countable collections} of intervals rather than \emph{sequences} of intervals (as in second-order RM).  In fact, Borel's proof of the Heine-Borel theorem in \cite{opborrelen2}*{p.\ 42} starts with: \emph{Let us enumerate our intervals, one after the other, according to whatever law, but determined}.  He then proceeds with the usual `interval-halving' proof, similar to Cousin in \cite{cousin1}.
\end{rem}

\subsection{Preliminaries}\label{prelim}
We introduce some axioms and definitions from (higher-order) RM needed below.  We refer to \cite{kohlenbach2}*{\S2} or \cite{dagsamIII}*{\S2} for the definition of Kohlebach's base theory $\RCAo$, and basic definitions like the real numbers $\R$ in $\RCAo$.  For completeness, some definitions are included in the technical appendix, namely Section \ref{app}. 

\subsubsection{Some axioms of higher-order arithmetic}\label{prelim1}
First of all, the functional $\varphi$ in $(\exists^{2})$ is clearly discontinuous at $f=11\dots$; in fact, $(\exists^{2})$ is equivalent to the existence of $F:\R\di\R$ such that $F(x)=1$ if $x>_{\R}0$, and $0$ otherwise (\cite{kohlenbach2}*{\S3}).  
\be\label{muk}\tag{$\exists^{2}$}
(\exists \varphi^{2}\leq_{2}1)(\forall f^{1})\big[(\exists n)(f(n)=0) \asa \varphi(f)=0    \big]. 
\ee
Related to $(\exists^{2})$, the functional $\mu^{2}$ in $(\mu^{2})$ is also called \emph{Feferman's $\mu$} (\cite{kohlenbach2}).
\begin{align}\label{mu}\tag{$\mu^{2}$}
(\exists \mu^{2})(\forall f^{1})\big[ (\exists n)(f(n)=0) \di [f(\mu(f))=0&\wedge (\forall i<\mu(f))(f(i)\ne 0) ]\\
& \wedge [ (\forall n)(f(n)\ne0)\di   \mu(f)=0]    \big].\notag
\end{align}
Intuitively, $\mu^{2}$ is the least-number-operator, i.e.\ $\mu(f)$ provides the least $n\in \N$ such that $f(n)=0$, if such there is.  
We have $(\exists^{2})\asa (\mu^{2})$ over $\RCAo$ and $\ACAo\equiv\RCAo+(\exists^{2})$ proves the same second-order sentences as $\ACA_{0}$ by \cite{hunterphd}*{Theorem~2.5}. 

\smallskip

Secondly, the Heine-Borel theorem states the existence of a finite sub-covering for an open covering of certain spaces. 
Now, a functional $\Psi:\R\di \R^{+}$ gives rise to the \emph{canonical cover} $\cup_{x\in I} I_{x}^{\Psi}$ for $I\equiv [0,1]$, where $I_{x}^{\Psi}$ is the open interval $(x-\Psi(x), x+\Psi(x))$.  
Hence, the uncountable covering $\cup_{x\in I} I_{x}^{\Psi}$ has a finite sub-covering by the Heine-Borel theorem; in symbols:
\begin{princ}[$\HBU$]
$(\forall \Psi:\R\di \R^{+})(\exists  y_{0}, \dots, y_{k}\in I){(\forall x\in I)}(x\in \cup_{i\leq k}I_{y_{i}}^{\Psi}).$
\end{princ}
Cousin and Lindel\"of formulate their covering theorems using canonical covers in \cites{cousin1, blindeloef}.
This restriction does not make much of a difference, as studied in \cite{sahotop}.

\smallskip

Thirdly, let $\WHBU$ be the following weakening of $\HBU$:
\begin{princ}[$\WHBU$] For any $\Psi:\R\di \R^{+} $ and $ \eps>_{\R}0$, there are pairwise distinct $ y_{0}, \dots, y_{k}\in I$ with $1-\eps <_{\R}\sum_{i\leq k}|J_{y_{i}}^{\Psi}| $, where $J_{y_{i+1}}^{\Psi}:= I_{y_{i+1}}^{\Psi}\setminus (\cup_{j\leq i}I_{y_{i}}^{\Psi}) $.
\end{princ}
As discussed at length in \cite{dagsamVI}, $\WHBU$ expresses the essence of the Vitali covering theorem for \emph{uncountable} coverings; Vitali already
considered the latter in \cite{vitaliorg}.  
Basic properties of the \emph{gauge integral} (\cite{zwette}) are equivalent to $\HBU$ while $\WHBU$ is equivalent to basic properties of the Lebesgue integral (without RM-codes; \cite{dagsamVI}).  
By \cite{dagsamIII, dagsamV,dagsamVI}, $\Z_{2}^{\Omega}$ proves $\HBU$ and $\WHBU$, but $\Z_{2}^{\omega}$ cannot.  The exact definition of $\Z_{2}^{\omega}$ and $\Z_{2}^{\Omega}$ is in the aforementioned references and Section~\ref{further}. What is relevant here is that $\Z_{2}^{\omega}$ and $\Z_{2}^{\Omega}$ are conservative extensions of $\Z_{2}$ by \cite{hunterphd}*{Cor.\ 2.6}, i.e.\ the former prove the same second-order sentences as the latter. 

\smallskip

We note that $\HBU$ (resp.\ $\WHBU$) is the higher-order counterpart of $\WKL$ (resp.\ $\WWKL$), i.e.\ \emph{weak K\"onig's lemma} (resp.\ \emph{weak weak K\"onig's lemma}) from RM as the $\ECF$-translation (\cites{troelstra1, kohlenbach2}) maps $\HBU$ (resp.\ $\WHBU$) to $\WKL$ (resp.\ $\WWKL$), i.e.\ these are (intuitively) \emph{weak} principles.  We refer to \cite{kohlenbach2}*{\S2} or Remark~\ref{ECF} for a discussion of the relation between $\ECF$ and $\RCAo$. 

\smallskip

Finally, the aforementioned results suggest that (higher-order) comprehension as in $\Z_{2}^{\omega}$ is not the right way of measuring the strength of $\HBU$. 
As a better alternative, we have introduced the following axiom in \cite{samph}.
\begin{princ}[$\BOOT$]
$(\forall Y^{2})(\exists X \subset \N)(\forall n^{0})\big[ n\in X \asa (\exists f^{1})(Y(f, n)=0)    \big]. $
\end{princ}
By \cite{samph}*{\S3}, $\BOOT$ is equivalent to convergence theorems for \emph{nets}, we have the implication $\BOOT\di \HBU$, and $\RCAo+\BOOT$ has the same first-order strength as $\ACA_{0}$.  
Moreover, $\BOOT$ is a natural fragment of Feferman's \emph{projection axiom} $\textsf{(Proj1)}$ from \cite{littlefef}.
Thus, $\BOOT$ is a natural axiom that provides a better `scale' for measuring the strength of $\HBU$ and its ilk, as discussed in \cite{samph, dagsamX}.

\subsubsection{Some basic definitions}\label{prelim2}
We introduce the higher-order definitions of `open' and `countable' set, as can be found in e.g.\ \cite{dagsamX, dagsamVII, dagsamXI}. 

\smallskip

First of all, open sets are represented in second-order RM as countable unions of basic open sets (\cite{simpson2}*{II.5.6}), and we refer to such sets as `RM-open'.
By \cite{simpson2}*{II.7.1}, one can effectively convert between RM-open sets and (RM-codes for) continuous characteristic functions.
Thus, a natural extension of the notion of `open set' is to allow \emph{arbitrary} (possibly discontinuous) characteristic functions, as is done in e.g.\ \cite{dagsamVII, dagsamX, samnetspilot}, which motivates the following definition.  
\bdefi[Sets in $\RCAo$]\label{openset}
We let $Y: \R \di \R$ represent subsets of $\R$ as follows: we write `$x \in Y$' for `$Y(x)>_{\R}0$' and call a set $Y\subseteq \R$ Ôopen' if for every $x \in Y$, there is an open ball $B(x, r) \subset Y$ with $r^{0}>0$.  
A set $Y$ is called `closed' if the complement, denoted $Y^{c}=\{x\in \R: x\not \in Y \}$, is open. 
\edefi
\noindent
For open $Y$ as in Definition \ref{openset}, the formula `$x\in Y$' has the same complexity (modulo higher types) as for RM-open sets, while given $(\exists^{2})$ it is equivalent to a `proper' characteristic function, only taking values `0' and `$1$'.  Hereafter, an `(open) set' refers to Definition \ref{openset}; `RM-open set' refers to the definition from second-order RM, as in e.g.\ \cite{simpson2}*{II.5.6}.

\smallskip

Secondly, the definition of `countable set' (Kunen; \cite{kunen}) is as follows in $\RCAo$. 
\bdefi[Countable subset of $\R$]\label{standard}~
A set $A\subseteq \R$ is \emph{countable} if there exists $Y:\R\di \N$ such that $(\forall x, y\in A)(Y(x)=_{0}Y(y)\di x=_{\R}y)$. 
If $Y:\R\di \N$ is also \emph{surjective}, i.e.\ $(\forall n\in \N)(\exists x\in A)(Y(x)=n)$, we call $A$ \emph{strongly countable}.
\edefi
Hereafter, `(strongly) countable' refers to Definition\ref{standard}, unless stated otherwise.  We note that `countable' is defined in second-order RM using \emph{sequences} (\cite{simpson2}*{V.4.2}), a notion we shall call `enumerable'.

\smallskip

Thirdly, we have explored the connection between $\HBU$, generalisations of $\HBU$, and fragments of the \emph{neighbourhood function principle} $\NFP$ from \cite{troeleke1} in \cites{samph, sahotop}.
In each case, nice equivalences were obtained \emph{assuming $\A_{0}$ as follows}. 
\begin{princ}[$\A_{0}$] For $Y^{2}$ and $A(\sigma)\equiv (\exists g\in 2^{\N})(Y(g, \sigma)=0)$, we have
\[
(\forall f\in \N^{\N})(\exists n\in \N)A(\overline{f}n)\di (\exists G^{2})(\forall f\in \N^{\N})A(\overline{f}G(f)), 
\]
where $\overline{f}n $ is the finite sequence $ \langle f(0), f(1), \dots, f(n-1) \rangle$.
\end{princ}
As discussed in \cites{sahotop, samph}, the axiom $\A_{0}$ is a fragment of $\NFP$ and can be viewed as a generalisation of $\QFAC^{1,0}$, included in $\RCAo$.  As an alternative to $\A_{0}$, one could add `extra data' or moduli to the theorems to be studied.  
\section{Main results}
In Section \ref{Y}, we show that $\HBU$ is equivalent to $\HBU_{\ae}$, i.e.\ the restriction to functions continuous almost everywhere , while the same restriction applied to $\WHBU$ results in a theorem equivalent to $\WWKL$ (see \cite{simpson2}*{X.1} for the latter).  In Section \ref{X}, we establish the splitting \eqref{baggin} involving $\HBU$.   
\subsection{Ontological parsimony and the Heine-Borel theorem}\label{Y}
We introduce $\HBU_{\text{\ae}}$, the restriction of $\HBU$ from Section \ref{prelim} to functions \emph{continuous almost everywhere}, and establish $\HBU\asa \HBU_{\ae}$ over $\RCAo$.
The same restriction for $\WHBU$ turns out to be equivalent to \emph{weak weak K\"onig's lemma} ($\WWKL$; see \cite{simpson2}*{X.1}), well-known from second-order RM. 

\smallskip

We first need the following definition, where we note that the usual\footnote{A set $A\subset \R$ is \emph{measure zero} if for any $\eps>0$ there is a sequence of basic open intervals $(I_{n})_{n\in \N}$ such that $\cup_{n\in \N}I_{n}$ covers $A$ and has total length below $\eps$.} definition of `measure zero' is used in RM.  
\bdefi[Continuity almost everywhere]\label{clier}
We say that $\Psi:[0,1]\di \R$ is \emph{continuous almost everywhere} if it is continuous outside of an RM-closed set $E\subset [0,1]$ which has measure zero.
\edefi
Let $\HBU_{\ae}$ be $\HBU$ restricted to functions continuous almost everywhere as in the previous definition.  The proof of the following theorem (in a stronger system) was suggested by Dag Normann, for which we are grateful.  

\begin{thm}\label{frigi}
The system $\RCAo$ proves $\HBU\asa \HBU_{\ae}$.
\end{thm}
\begin{proof}
First of all, as noted in Section \ref{prelim}, $(\exists^{2})$ is equivalent to the existence of a discontinuous $\R\di \R$-function, namely by \cite{kohlenbach2}*{Prop.\ 3.12}.
Thus, in case $\neg(\exists^{2})$, all functions on $\R$ are continuous.  In this case, we trivially obtain $\HBU\asa \HBU_{\ae}$.  Since $\RCAo$ is a classical system, we
have the law of excluded middle as in $\neg(\exists^{2})\vee (\exists^{2})$.  As we have provided a proof in the first case $\neg(\exists^{2})$, it suffices to provide a proof assuming $(\exists^{2})$, and the law of excluded middle finishes the proof.  
Hence, for the rest of the proof, we may assume $(\exists^{2})$.

\smallskip

Secondly, the \emph{Cantor middle third set} $\mathcal{C}\subset [0,1]$ is available in $\RCA_{0}$ by (the proof of) \cite{simpson2}*{IV.1.2} as an RM-closed set, as well as the well-known recursive homeomorphism from Cantor space $2^{\N}$ to $\mathcal{C}$ defined as $H:2^{\N}\di [0,1]$ and $H(f):=\sum_{n=0}^{\infty}\frac{2 f(n)}{3^{n+1}}$.   Note that given $\exists^{2}$, we can decide whether $x\in \mathcal{C}$ or not.  

\smallskip

Thirdly, we prove $\HBU_{\ae}\di \HBU_{\cc}$, where the latter is $\HBU$ for $2^{\N}$ as follows:  
\be\tag{$\HBU_{\cc}$}
(\forall G^{2})(\exists f_{0}, \dots, f_{k}\in 2^{\N})(\forall g\in 2^{\N})(\exists i\leq k)(g\in [\overline{f_{i}}G(f_{i})]  )
\ee
and where $[\sigma]$ is the open neighbourhood in $2^{\N}$ of sequences starting with the finite binary sequence $\sigma$.
The equivalence $\HBU\asa \HBU_{\cc}$ may be found in \cites{dagsamIII, dagsamV}.
Now assume $\HBC_{\ae}$ and fix $G^{2}$ and define $\Psi:[0,1]\di \R^{+}$ using $(\exists^{2})$ as:
\be\label{bydef}
\Psi(x):=
\begin{cases}
d(x, \mathcal{C})  & x\not\in \mathcal{C}\\
\frac{1}{2^{G(I(x))}} & \textup{ otherwise } 
\end{cases},
\ee
where $I(x)$ is the unique $f\in 2^{\N}$ such that $H(f)=x$ in case $x\in \mathcal{C}$, and $00\dots$ otherwise.
Note that the distance function $d(x, \mathcal{C})$ exists given $\ACA_{0}$ by \cite{withgusto}*{Theorem 1.2}.
Clearly, $\exists^{2}$ allows us to define this function as a third-order object that is continuous on $[0,1]\setminus\mathcal{C}$.
Since $\mathcal{C}$ has measure zero (and is RM-closed), apply $\HBU_{\ae}$ to $\cup_{x\in [0,1]}I_{x}^{\Psi}$.  
Let $y_{0}, \dots, y_{k}$ be such that $\cup_{i\leq k}I_{y_{i}}^{\Psi}$ covers $[0,1]$.  
By the definition of $\Psi$ in \eqref{bydef}, if $x\in [0,1]\setminus \mathcal{C}$, then $\mathcal{C}\cap I_{x}^{\Psi}=\emptyset$.
Hence, let $z_{0}, \dots, z_{m}$ be those $y_{i}\in \mathcal{C}$ for $i\leq k$ and note that $\cup_{j\leq m}I_{z_{j}}^{\Psi}$ covers $\mathcal{C}$.
Clearly, $I(z_{0}), \dots, I(z_{m})$ yields a finite sub-cover of $\cup_{f\in 2^{\N}}[\overline{f}G(f)]$, and $\HBU_{\cc}$ follows. 
%
%
%
%
%
\qed \end{proof}
We could of course formulate $\HBU_{\ae}$ with the higher-order notion of `closed set' from \cite{dagsamVII}, and the equivalence from the theorem would still go through. 
The proof of the theorem also immediately yields the following. 
\begin{cor}[$\ACAo$] $\HBU$ is equivalent to the Heine-Borel theorem for canonical coverings $\cup_{x\in E}I_{x}^{\Psi}$, where $E\subset [0,1]$ is RM-closed and has measure zero. 
\end{cor}
As expected, Theorem \ref{frigi} generalises to principles that imply $\HBU$ over $\RCAo$ (see \cite{dagsamX}*{Figure 1} for an overview) and that boast a third-order functional to which the `continuous almost everywhere' restriction can be naturally applied.
An example is the following corollary involving $\BOOT$.  
\begin{cor}
The system $\RCAo$ proves $\BOOT\asa \BOOT_{\ae}$, where the latter is 
\[
(\exists X \subset \N)(\forall n^{0})\big[ n\in X \asa (\exists x\in [0,1])(Y(x, n)=0)    \big],
\]
where $\lambda x.Y(x, n)$ is continuous almost everywhere on $[0,1]$ for any fixed $n\in \N$.
\end{cor}
\begin{proof}
In case $\neg(\exists^{2})$, all functions on $\R$ are continuous by \cite{kohlenbach2}*{Prop.\ 3.12}; in this case, the equivalence is trivial.  
In case $(\exists^{2})$, the forward direction is immediate, modulo coding real numbers given $\exists^{2}$.  For the reverse direction, fix $Y^{2}$ and note that we may restrict the quantifier $(\exists f^{1})$ in $\BOOT$ to $2^{\N}$ without loss of generality.  Indeed, $\mu^{2}$ allows us to represent $f^{1}$ via its graph, a subset of $\N^{2}$, which can be coded as a binary sequence.   
Now define
\be\label{lok}
Z(x, n):=
\begin{cases}
0 & x\in \mathcal{C} \wedge Y(I(x), n)=0  \\
1  & \textup{ otherwise }
\end{cases}, 
\ee
where $\mathcal{C}$ and $I$ are as in the theorem.  
Note that $\lambda x.Z(x, n)$ is continuous outside of $\mathcal{C}$. 
By $\BOOT_{\ae}$, there is $X\subset \N$ such that for all $n\in \N$, we have:
\[
n\in X\asa (\exists x\in [0,1])(Z(x, n)=0)\asa (\exists f\in 2^{\N})(Y(f, n)=0),
\]
where the last equivalence is by the definition of $Z$ in \eqref{lok}.  
\qed\end{proof}
Next, we show that the \emph{Vitali covering theorem} as in $\WHBU$ behaves quite differently from the Heine-Borel theorem as in $\HBU$.  
Recall that the Heine-Borel theorem applies to open coverings of compact sets, while the Vitali covering theorem applies to Vitali coverings\footnote{An open covering $V$ is a \emph{Vitali covering} of $E$ if any point of $E$ can be covered by some open in $V$ with arbitrary small (Lebesgue) measure.} of any set $E$ of finite (Lebesgue) measure.  
The former provides a finite sub-covering while the latter provides a sequence that covers $E$ up to a set of measure zero.  As argued in \cite{dagsamVI}, $\WHBU$ is the combinatorial essence of Vitali's covering theorem.  

\smallskip

Now, let $\WHBU_{\ae}$ be $\WHBU$ restricted to functions continuous almost everywhere, as in Definition~\ref{clier}; recall that $\Z_{2}^{\omega}$ cannot prove $\WHBU$. 
\begin{thm}
The system $\RCAo+\WKL$ proves $\WHBU_{\ae}$.
\end{thm}
\begin{proof}
Let $\Psi:[0,1]\di \R^{+}$ be continuous on $[0,1]\setminus E$ with $E\subset [0,1]$ of measure zero and RM-closed.  Fix $\eps>0$ and let $\cup_{n\in \N}I_{n}$ be a union of basic open intervals covering $E$ and with measure at most $\eps/2$.
Then $[0,1]$ is covered by:  
\be\label{blaffer}
\cup_{q\in \Q\setminus E}B(q, \Psi(q))\bigcup \cup_{n\in \N}I_{n}.
\ee 
Indeed, that the covering in \eqref{blaffer} covers $E$ is trivial, while $[0,1]\setminus E$ is (RM)-open.  Hence, $x_{0} \in[0,1]\setminus E$ implies that $B(x_{0}, r)\subset [0,1]\setminus E$ for $r>0$ small enough and for $q\in \Q\cap [0,1] $ close enough to $x_{0}$, we have $x_{0}\in B(q, \Psi(q))$.
By \cite{simpson2}*{IV.1}, $\WKL$ is equivalent to the countable Heine-Borel theorem.
Hence, there are $q_{0}, \dots, q_{k}\in \Q\setminus E$ and $n_{0}\in \N$
such that the finite union $\cup_{i=1}^{k}B(q_{i}, \Psi(q_{i}))\bigcup\cup_{j=0}^{n_{0}}I_{j}$ covers $[0,1]$. 
Since the measure of $\cup_{j=0}^{n_{0}}I_{j}$ is at most $\eps/2$, the measure of $\cup_{i=1}^{k}B(q_{i}, \Psi(q_{i}))$ is at least $1-\eps/2$, as required by $\WHBU_{\ae}$. 
\qed \end{proof}
\begin{cor}
The system $\RCAo$ proves $\WWKL\asa \WHBU_{\ae}$.
\end{cor}
\begin{proof}
The reverse implication is immediate in light of the RM of $\WWKL$ in \cite{simpson2}*{X.1}, which involves the Vitali covering theorem for countable coverings (given by a sequence).
For the forward implication, convert the cover from \eqref{blaffer} to a Vitali cover and use \cite{simpson2}*{X.1.13}. 
\qed \end{proof}
Finally, recall Remark \ref{fundies} discussing the foundational significance of the above.

\subsection{Splittings for the Heine-Borel theorem}\label{X}
We establish a splitting for $\HBU$ as in Theorem \ref{forgo} based on known principles formulated with \emph{countable sets} as in Definition \ref{standard}.  
As will become clear, there is also some historical interest in this study.

\smallskip

First of all, the following principle $\HBC_{0}$ is studied in \cite{dagsamX}*{\S3}, while the (historical and foundational) significance of this principle is discussed in Remark~\ref{fundies}.
The aforementioned system $\Z_{2}^{\omega}$ cannot prove $\HBC_{0}$.
\begin{princ}[$\HBC_{0}$]
For countable $A\subset \R^{2}$ with $(\forall x\in [0,1])(\exists (a, b)\in A)(x\in (a, b))$, there are $(a_{0}, b_{0}), \dots, (a_{k}, b_{k})\in A$ with $(\forall x\in [0,1])(\exists i\leq k)(x\in (a_{i},b_{i} ))$.
\end{princ}
%
%
Secondly, the second-order Vitali covering theorem has a number of equivalent formulations (see \cite{simpson2}*{X.1}), including the statement \emph{a countable covering of $[0,1]$ has a sub-collection with measure zero complement}.
Intuitively speaking, the following principle $\WHBU^{+}$ strengthens `measure zero' to `countable'.  Alternatively, $\WHBU^{+}$ can be viewed as a weakening of the Lindel\"of lemma, introduced in \cite{blindeloef} and studied in higher-order RM in \cites{dagsamIII, dagsamV}.   
\begin{princ}[$\WHBU^{+}$]
For $\Psi:[0,1]\di \R^{+}$, there is a sequence $(y_{n})_{n\in \N}$ in $[0,1]$ such that $[0,1]\setminus \cup_{n\in \N}I_{y_{n}}^{\Psi}$ is countable.
\end{princ}
Note that $\WHBU^{+}+\HBC_{0}$ yields a conservative\footnote{The system $\RCAo+\neg(\exists^{2})$ is an $\L_{2}$-conservative extension of $\RCAo$ and the former readily proves $\WHBU^{+}+\HBC_{0}$.  By constrast $\HBU\di \WKL$ over $\RCAo$.} extension of $\RCAo$, i.e.\ the former cannot imply $\HBU$ without the presence of $\WKL$.  Other independence results are provided by Theorem \ref{indie}. 

\smallskip

We have the following theorem, where $\A_{0}$ was introduced in Section \ref{prelim2}.
This axiom can be avoided 
by enriching\footnote{In particular, one would add a function $G:[0,1]\di \R^{2}$ to the antecedent of $\HBC_{0}$ such that $G(x)\in A$ and $x\in \big(G(x)(1), G(x)(2)\big)$ for $x\in [0,1]$.  In this way, the covering is given by $\cup_{x\in [0,1]}(G(x)(1), G(x)(2))$.} the antecedent of $\HBC_{0}$.  
\begin{thm}\label{forgo}
The system $\RCAo+\A_{0}$ proves 
\be\label{fung}
[\WHBU^{+}+\HBC_{0}+\WKL]\asa \HBU,
\ee
where the axiom $\A_{0}$ is only needed for $\HBU\di \HBC_{0}$.
\end{thm}
\begin{proof}
First of all, in case $\neg(\exists^{2})$, all functions on $\R$ are continuous, rendering $\WHBU^{+}+\HBC_{0}$ trivial while $\HBU$ reduces to $\WKL$.
Hence, for the rest of the proof, we may assume $(\exists^{2})$, by the law of excluded middle as in $(\exists^{2})\vee \neg(\exists^{2})$.

\smallskip

For the reverse implication, assume $\A_{0}+\HBU$ and let $A$ be as in $\HBC_{0}$.  
The functional $\exists^{2}$ can uniformly convert real numbers to a binary representation.  
Hence \eqref{typical} is equivalent to a formula as in the antecedent of $\A_{0}$:
\be\label{typical}\textstyle
(\forall x\in [0,1])(\exists n\in \N)\big[ (\exists (a, b)\in A)( a< [x](n+1)-\frac{1}{2^{n}} \wedge  [x](n+1)+\frac{1}{2^{n}} <b  ) \big], 
\ee
where `$[x](n)$' is the $n$-th approximation of the real $x$, given as a fast-converging Cauchy sequence. 
Apply $\A_{0}$ to \eqref{typical} to obtain $G:[0,1]\di \N$ such that $G(x)=n$ as in \eqref{typical}.  Apply $\HBU$ to $\cup_{x\in [0,1]}I^{\Psi}_{x}$ for $\Psi(x):=\frac{1}{2^{G(x)}}$.
The finite sub-cover $y_{0}, \dots, y_{k}\in [0,1]$ provided by $\HBU$ gives rise to $(a_{i}, b_{i})\in A$ containing $I_{y_{i}}^{\Psi}$ for $i\leq k$ by the definition of $G$.  
Moreover, $\HBU$ implies $\WKL$ as the latter is equivalent to the `countable' Heine-Borel theorem as in \cite{simpson2}*{IV.1}.
Clearly, the empty set is countable by Definition \ref{standard} and $\HBU\di \WHBU^{+}$ is therefore trivial. 

\smallskip

For the forward implication, fix $\Psi:[0,1]\di \R^{+}$ and let $(y_{n})_{n\in \N}$ be as in $\WHBU^{+}$.  
Define `$x\in B$' as $x\in [0,1]\setminus \cup_{n\in \N}I_{y_{n}}^{\Psi}$ and note that when $B$ is empty, the theorem follows as $\WKL$ implies the second-order Heine-Borel theorem (\cite{simpson2}*{IV.1}).
Now assume $B\ne \emptyset$ and define $A$ as the set of $(a, b)$ such that either $(a, b)= I_{x}^{\Psi}$ for $x\in B$, or $(a, b)=I_{y_{n}}^{\Psi}$ for some $n\in \N$.
Note that in the first case, $(a, b)\in A$ if and only $\frac{a+b}{2}\in B$, i.e.\ defining $A$ does not require quantifying over $\R$.
Moreover, $A$ is countable because $B$ is: if $Y$ is injective on $B$, then $W$ defined as follows is injective on $A$:
\[   \textstyle
W\big( (a,b) \big):=
\begin{cases}
2Y(\frac{a+b}{2}) & \frac{a+b}{2}\in B \\
H((a, b)) & \textup{otherwise} 
\end{cases},
\]
where $H((a,b))$ is the least $n\in \N$ such that $(a, b)=I_{y_{n}}^{\Psi}$, if such there is, and zero otherwise. 
The intervals in the set $A$ cover $[0,1]$ as in the antecedent of $\HBC_{0}$, and the latter now implies $\HBU$.
\qed \end{proof}
The principles $\WHBU^{+}$ and $\HBC_{0}$ are `quite' independent by the following theorem, assuming the systems therein are consistent. 
\begin{thm}\label{indie}
The system $\Z_{2}^{\omega}+\QFAC^{0,1}+\WHBU^{+}$ cannot prove $\HBC_{0}$. \\
The system $\RCAo+\HBC_{0}+\WHBU^{+}$ cannot prove $\WKL_{0}$.
\end{thm}
\begin{proof}
For the first part, suppose $\Z_{2}^{\omega}+\QFAC^{0,1}+\WHBU^{+}$ does prove $\HBC_{0}$.  The latter implies $\NIN$ as follows by \cite{dagsamX}*{Cor.\ 3.2}:
\be\tag{$\NIN$}
(\forall Y:[0,1]\di \N)(\exists x, y\in [0,1])(Y(x)=Y(y)\wedge x\ne_{\R}y).
\ee
Clearly, $\neg\NIN$ implies $\WHBU^{+}$, and we obtain that $\Z_{2}^{\omega}+\QFAC^{0,1}+\neg\NIN$ proves a contradiction, namely $\WHBU^{+}$ and its negation.
Hence, $\Z_{2}^{\omega}+\QFAC^{0,1}$ proves $\NIN$, a contradiction by \cite{dagsamX}*{Theorem 3.1}, and the first part follows.

\smallskip

For the second part, the $\ECF$-translation (see Remark \ref{ECF}) converts $\HBC_{0}+\WHBU^{+}$ into a triviality.
\qed \end{proof}
Finally, we discuss similar results as follows.  Of course, the proof of Theorem~\ref{forgo} goes through \emph{mutatis mutandis} for $\WHBU^{+}+\HBC_{0}$ formulated using \emph{strongly} countable sets.   
Moreover, \eqref{toch} can be proved in the same way as \eqref{fung}, assuming additional countable choice as in $\QFAC^{0,1}$:
\be\label{toch}
\WHBU \asa [\WHBU^{+}+\WHBC_{0}+\WWKL],
\ee
where $\WHBC_{0}$ is $\HBC_{0}$ with the conclusion weakened to the existence of a sequence $(a_{n}, b_{n})_{n\in \N}$ of intervals in $A$ with measure at least one. 
Also, if we generalise $\HBU$ to coverings of any \emph{separably closed} set in $[0,1]$, the resulting version of \eqref{fung} involves $\ACA_{0}$ rather than $\WKL_{0}$ in light of \cite{hirstrm2001}*{Theorem 2}.

%
%
%
%

%


\appendix

\section{Reverse Mathematics: second- and higher-order}\label{app}

\subsection{Reverse Mathematics}\label{kapp}
Reverse Mathematics (RM hereafter) is a program in the foundations of mathematics initiated around 1975 by Friedman (\cites{fried,fried2}) and developed extensively by Simpson (\cite{simpson2}).  
The aim of RM is to identify the minimal axioms needed to prove theorems of ordinary, i.e.\ non-set theoretical, mathematics. 
We refer to \cite{stillebron} for a basic introduction to RM and to \cite{simpson2, simpson1} for an overview of RM.  
The details of Kohlenbach's \emph{higher-order} RM may be found in \cite{kohlenbach2}, including the base theory $\RCAo$.  
The latter is connected to $\RCA_{0}$ by the $\ECF$-translation as follows. 
\begin{rem}[The $\ECF$-interpretation]\label{ECF}\rm
The (rather) technical definition of $\ECF$ may be found in \cite{troelstra1}*{p.\ 138, \S2.6}.
Intuitively, the $\ECF$-interpretation $[A]_{\ECF}$ of a formula $A\in \L_{\omega}$ is just $A$ with all variables 
of type two and higher replaced by type one variables ranging over so-called `associates' or `RM-codes'; the latter are (countable) representations of continuous functionals.  
The $\ECF$-interpretation connects $\RCAo$ and $\RCA_{0}$ (see \cite{kohlenbach2}*{Prop.\ 3.1}) in that if $\RCAo$ proves $A$, then $\RCA_{0}$ proves $[A]_{\ECF}$, again `up to language', as $\RCA_{0}$ is 
formulated using sets, and $[A]_{\ECF}$ is formulated using types, i.e.\ using type zero and one objects.  
\end{rem}
In light of the widespread use of codes in RM and the common practise of identifying codes with the objects being coded, it is no exaggeration to refer to $\ECF$ as the \emph{canonical} embedding of higher-order into second-order arithmetic. 

\smallskip

We now introduce the usual notations for common mathematical notions.  
\begin{defi}[Real numbers and related notions in $\RCAo$]\label{keepintireal}\rm~
\begin{enumerate}
 \renewcommand{\theenumi}{\alph{enumi}}
\item Natural numbers correspond to type zero objects, and we use `$n^{0}$' and `$n\in \N$' interchangeably.  Rational numbers are defined as signed quotients of natural numbers, and `$q\in \Q$' and `$<_{\Q}$' have their usual meaning.    
\item Real numbers are coded by fast-converging Cauchy sequences $q_{(\cdot)}:\N\di \Q$, i.e.\  such that $(\forall n^{0}, i^{0})(|q_{n}-q_{n+i}|<_{\Q} \frac{1}{2^{n}})$.  
We use Kohlenbach's `hat function' from \cite{kohlenbach2}*{p.\ 289} to guarantee that every $q^{1}$ defines a real number.  
\item We write `$x\in \R$' to express that $x^{1}:=(q^{1}_{(\cdot)})$ represents a real as in the previous item and write $[x](k):=q_{k}$ for the $k$-th approximation of $x$.    
\item Two reals $x, y$ represented by $q_{(\cdot)}$ and $r_{(\cdot)}$ are \emph{equal}, denoted $x=_{\R}y$, if $(\forall n^{0})(|q_{n}-r_{n}|\leq {2^{-n+1}})$. Inequality `$<_{\R}$' is defined similarly.  
We sometimes omit the subscript `$\R$' if it is clear from context.           
\item Functions $F:\R\di \R$ are represented by $\Phi^{1\di 1}$ mapping equal reals to equal reals, i.e.\ extensionality as in $(\forall x , y\in \R)(x=_{\R}y\di \Phi(x)=_{\R}\Phi(y))$.\label{EXTEN}
\item Binary sequences are denoted `$f,g\in C$' or `$f, g\in 2^{\N}$'.  Elements of Baire space are given by $f^{1}, g^{1}$, but also denoted `$f, g\in \N^{\N}$'.
\end{enumerate}
\end{defi}
\begin{nota}[Finite sequences]\label{skim}\rm
The type for `finite sequences of objects of type $\rho$' is denoted $\rho^{*}$, which we shall only use for $\rho=0,1$.  
Since the usual coding of pairs of numbers goes through in $\RCAo$, we shall not always distinguish between $0$ and $0^{*}$. 
Similarly, we assume a fixed coding for finite sequences of type $1$ and shall make use of the type `$1^{*}$'.  
In general, we do not always distinguish between `$s^{\rho}$' and `$\langle s^{\rho}\rangle$', where the former is `the object $s$ of type $\rho$', and the latter is `the sequence of type $\rho^{*}$ with only element $s^{\rho}$'.  The empty sequence for the type $\rho^{*}$ is denoted by `$\langle \rangle_{\rho}$', usually with the typing omitted.  
Furthermore, we denote by `$|s|=n$' the length of the finite sequence $s^{\rho^{*}}=\langle s_{0}^{\rho},s_{1}^{\rho},\dots,s_{n-1}^{\rho}\rangle$, where $|\langle\rangle|=0$, i.e.\ the empty sequence has length zero.  For sequences $s^{\rho^{*}}, t^{\rho^{*}}$, we denote by `$s*t$' the concatenation of $s$ and $t$, i.e.\ $(s*t)(i)=s(i)$ for $i<|s|$ and $(s*t)(j)=t(|s|-j)$ for $|s|\leq j< |s|+|t|$. For a sequence $s^{\rho^{*}}$, we define $\overline{s}N:=\langle s(0), s(1), \dots,  s(N-1)\rangle $ for $N^{0}<|s|$.  
For a sequence $\alpha^{0\di \rho}$, we also write $\overline{\alpha}N=\langle \alpha(0), \alpha(1),\dots, \alpha(N-1)\rangle$ for \emph{any} $N^{0}$.  Finally, 
$(\forall q^{\rho}\in Q^{\rho^{*}})A(q)$ abbreviates $(\forall i^{0}<|Q|)A(Q(i))$, which is (equivalent to) quantifier-free if $A$ is.   
\end{nota}
\subsection{Further systems}\label{further}
We define some standard higher-order systems that constitute the counterpart of e.g.\ $\FIVE$ and $\Z_{2}$.
First of all, \emph{the Suslin functional} $\SS^{2}$ is defined in \cite{kohlenbach2} as:
\be\tag{$\SS^{2}$}
(\exists\SS^{2}\leq_{2}1)(\forall f^{1})\big[  (\exists g^{1})(\forall n^{0})(f(\overline{g}n)=0)\asa \SS(f)=0  \big].
\ee
The system $\FIVE^{\omega}\equiv \RCAo+(\SS^{2})$ proves the same $\Pi_{3}^{1}$-sentences as $\FIVE$ by \cite{yamayamaharehare}*{Theorem 2.2}.   
By definition, the Suslin functional $\SS^{2}$ can decide whether a $\Sigma_{1}^{1}$-formula as in the left-hand side of $(\SS^{2})$ is true or false.   We similarly define the functional $\SS_{k}^{2}$ which decides the truth or falsity of $\Sigma_{k}^{1}$-formulas from $\L_{2}$; we also define 
the system $\SIXK$ as $\RCAo+(\SS_{k}^{2})$, where  $(\SS_{k}^{2})$ expresses that $\SS_{k}^{2}$ exists.  
We note that the operators $\nu_{n}$ from \cite{boekskeopendoen}*{p.\ 129} are essentially $\SS_{n}^{2}$ strengthened to return a witness (if existant) to the $\Sigma_{n}^{1}$-formula at hand.  

\smallskip

\noindent
Secondly, second-order arithmetic $\Z_{2}$ readily follows from $\cup_{k}\SIXK$, or from:
\be\tag{$\exists^{3}$}
(\exists E^{3}\leq_{3}1)(\forall Y^{2})\big[  (\exists f^{1})(Y(f)=0)\asa E(Y)=0  \big], 
\ee
and we therefore define $\Z_{2}^{\Omega}\equiv \RCAo+(\exists^{3})$ and $\Z_{2}^\omega\equiv \cup_{k}\SIXK$, which are conservative over $\Z_{2}$ by \cite{hunterphd}*{Cor.\ 2.6}. 
Despite this close connection, $\Z_{2}^{\omega}$ and $\Z_{2}^{\Omega}$ can behave quite differently, as discussed in e.g.\ \cite{dagsamIII}*{\S2.2}.   
The functional from $(\exists^{3})$ is also called `$\exists^{3}$', and we use the same convention for other functionals.

%

\section{References}
\begin{biblist}


\bib{opborrelen2}{book}{
  author={Borel, E.},
  title={Le\c {c}ons sur la th\'eorie des fonctions},
  year={1898},
  publisher={Gauthier-Villars, Paris},
  pages={pp.\ 136},
}

\bib{boekskeopendoen}{book}{
   author={Buchholz, Wilfried},
   author={Feferman, Solomon},
   author={Pohlers, Wolfram},
   author={Sieg, Wilfried},
   title={Iterated inductive definitions and subsystems of analysis},
   series={LNM 897},
   publisher={Springer},
   date={1981},
   pages={v+383},
}


\bib{cousin1}{article}{
  author={Cousin, P.},
  title={Sur les fonctions de $n$ variables complexes},
  journal={Acta Math.},
  volume={19},
  date={1895},
  pages={1--61},
}

\bib{littlefef}{book}{
  author={Feferman, Solomon},
  title={How a Little Bit goes a Long Way: Predicative Foundations of Analysis},
  year={2013},
  note={unpublished notes from 1977-1981 with updated introduction, \url {https://math.stanford.edu/~feferman/papers/pfa.pdf}},
}

\bib{fried}{article}{
  author={Friedman, Harvey},
  title={Some systems of second order arithmetic and their use},
  conference={ title={Proceedings of the ICM (Vancouver, B.\ C., 1974), Vol.\ 1}, },
  book={ },
  date={1975},
  pages={235--242},
}

\bib{fried2}{article}{
  author={Friedman, Harvey},
  title={ Systems of second order arithmetic with restricted induction, I \& II \(Abstracts\) },
  journal={Journal of Symbolic Logic},
  volume={41},
  date={1976},
  pages={557--559},
}

\bib{withgusto}{article}{
  author={Giusto, Mariagnese},
  author={Simpson, Stephen G.},
  title={Located sets and reverse mathematics},
  journal={J. Symbolic Logic},
  volume={65},
  date={2000},
  number={3},
  pages={1451--1480},
}

\bib{hirstrm2001}{article}{
  author={Hirst, Jeffry L.},
  title={A note on compactness of countable sets},
  conference={ title={Reverse mathematics 2001}, },
  book={ series={Lect. Notes Log.}, volume={21}, publisher={Assoc. Symbol. Logic}, },
  date={2005},
  pages={219--221},
}

\bib{hunterphd}{book}{
  author={Hunter, James},
  title={Higher-order reverse topology},
  note={Thesis (Ph.D.)--The University of Wisconsin - Madison},
  publisher={ProQuest LLC, Ann Arbor, MI},
  date={2008},
  pages={97},
}

\bib{kohlenbach2}{article}{
  author={Kohlenbach, Ulrich},
  title={Higher order reverse mathematics},
  conference={ title={Reverse mathematics 2001}, },
  book={ series={Lect. Notes Log.}, volume={21}, publisher={ASL}, },
  date={2005},
  pages={281--295},
}


\bib{kunen}{book}{
  author={Kunen, Kenneth},
  title={Set theory},
  series={Studies in Logic},
  volume={34},
  publisher={College Publications, London},
  date={2011},
  pages={viii+401},
}

\bib{blindeloef}{article}{
  author={Lindel\"of, Ernst},
  title={Sur Quelques Points De La Th\'eorie Des Ensembles},
  journal={Comptes Rendus},
  date={1903},
  pages={697--700},
}


\bib{dagsamIII}{article}{
  author={Normann, Dag},
  author={Sanders, Sam},
  title={On the mathematical and foundational significance of the uncountable},
  journal={Journal of Mathematical Logic, \url {https://doi.org/10.1142/S0219061319500016}},
  date={2019},
}

\bib{dagsamVI}{article}{
  author={Normann, Dag},
  author={Sanders, Sam},
  title={Representations in measure theory},
  journal={Submitted, arXiv: \url {https://arxiv.org/abs/1902.02756}},
  date={2019},
}

\bib{dagsamVII}{article}{
  author={Normann, Dag},
  author={Sanders, Sam},
  title={Open sets in Reverse Mathematics and Computability Theory},
  journal={Journal of Logic and Computability},
  volume={30},
  number={8},
  date={2020},
  pages={pp.\ 40},
}

\bib{dagsamV}{article}{
  author={Normann, Dag},
  author={Sanders, Sam},
  title={Pincherle's theorem in reverse mathematics and computability theory},
  journal={Ann. Pure Appl. Logic},
  volume={171},
  date={2020},
  number={5},
  pages={102788, 41},
}

\bib{dagsamX}{article}{
  author={Normann, Dag},
  author={Sanders, Sam},
  title={On the uncountability of $\mathbb {R}$},
  journal={Submitted, arxiv: \url {https://arxiv.org/abs/2007.07560}},
  pages={pp.\ 37},
  date={2020},
}

\bib{dagsamIX}{article}{
   author={Normann, Dag},
   author={Sanders, Sam},
   title={The Axiom of Choice in Computability Theory and Reverse Mathematics},
   journal={Journal of logic and computation},
   volume={31},
   number={1},
   pages={297--325},
   date={2021},
}

\bib{dagsamXI}{article}{
   author={Normann, Dag},
   author={Sanders, Sam},
   title={On robust theorems due to Bolzano, Weierstrass, and Cantor in Reverse Mathematics},
   journal={See \url{https://arxiv.org/abs/2102.04787}},
   pages={pp.\ 30},
   date={2021},
}


\bib{yamayamaharehare}{article}{
   author={Sakamoto, Nobuyuki},
   author={Yamazaki, Takeshi},
   title={Uniform versions of some axioms of second order arithmetic},
   journal={MLQ Math. Log. Q.},
   volume={50},
   date={2004},
   number={6},
   pages={587--593},
}


\bib{samnetspilot}{article}{
  author={Sanders, Sam},
  title={Nets and Reverse Mathematics: a pilot study},
  year={2019},
  journal={Computability, \url {doi: 10.3233/COM-190265}},
  pages={pp.\ 34},
}

\bib{samph}{article}{
  author={Sanders, Sam},
  title={Plato and the foundations of mathematics},
  year={2019},
  journal={Submitted, arxiv: \url {https://arxiv.org/abs/1908.05676}},
  pages={pp.\ 40},
}

\bib{samsplit}{article}{
  author={Sanders, Sam},
  title={Splittings and disjunctions in reverse mathematics},
  journal={Notre Dame J. Form. Log.},
  volume={61},
  date={2020},
  number={1},
  pages={51--74},
}

\bib{sahotop}{article}{
  author={Sanders, Sam},
  title={Reverse Mathematics of topology: dimension, paracompactness, and splittings},
  year={2020},
  journal={Notre Dame Journal for Formal Logic},
  pages={537-559},
  volume={61},
  number={4},
}

\bib{simpson1}{collection}{
  title={Reverse mathematics 2001},
  series={Lecture Notes in Logic},
  volume={21},
  editor={Simpson, Stephen G.},
  publisher={ASL},
  date={2005},
  pages={x+401},
}

\bib{simpson2}{book}{
  author={Simpson, Stephen G.},
  title={Subsystems of second order arithmetic},
  series={Perspectives in Logic},
  edition={2},
  publisher={Cambridge University Press},
  date={2009},
  pages={xvi+444},
}

\bib{stillebron}{book}{
  author={Stillwell, J.},
  title={Reverse mathematics, proofs from the inside out},
  pages={xiii + 182},
  year={2018},
  publisher={Princeton Univ.\ Press},
}

\bib{zwette}{book}{
  author={Swartz, Charles},
  title={Introduction to gauge integrals},
  publisher={World Scientific},
  date={2001},
  pages={x+157},
}

\bib{troelstra1}{book}{
  author={Troelstra, Anne Sjerp},
  title={Metamathematical investigation of intuitionistic arithmetic and analysis},
  note={Lecture Notes in Mathematics, Vol.\ 344},
  publisher={Springer Berlin},
  date={1973},
  pages={xv+485},
}

\bib{troeleke1}{book}{
  author={Troelstra, Anne Sjerp},
  author={van Dalen, Dirk},
  title={Constructivism in mathematics. Vol. I},
  series={Stud. in Logic and the Found. of Math.},
  volume={121},
  publisher={North-Holland},
  date={1988},
  pages={xx+342+XIV},
}



\bib{vitaliorg}{article}{
  author={Vitali, Guiseppe},
  title={Sui gruppi di punti e sulle funzioni di variabili reali.},
  journal={Atti della Accademia delle Scienze di Torino, vol XLIII},
  date={1907},
  number={4},
  pages={229--247},
}


\end{biblist}
%
%
%
%
%
\end{document}